\documentclass[12pt]{amsart}
\include{epic.sty}
\usepackage{latexsym,amssymb}

\usepackage{amsmath}
\usepackage{amsthm}
\usepackage{multirow}

\usepackage[latin1]{inputenc}
\newlength{\originalbase}
\newcommand{\spacing}[1]{\setlength{\baselineskip}{#1\originalbase}}
\spacing{1}

\newcommand{\aHC}{\mathfrak{H}^{\mathfrak c}}           
\newcommand{\coaHC}{\overline{\mathfrak{H}}^{\mathsf c}}
\newcommand{\GRaHC}{\mathfrak{H}^0}                        
\newcommand{\coGRaHC}{\overline{\mathfrak{H}}^0}
\newcommand{\saH}{\mathfrak{H}^{\mathsf{sp}}}          
\newcommand{\cosaH}{\overline{\mathfrak{H}}^{\mathsf{sp}}}
\newcommand{\saHo}{\mathfrak{H}^{-,0}}
\newcommand{\cosaHo}{\overline{\mathfrak{H}}^{\mathsf{sp},0}}
\newcommand{\Tr}{\operatorname{Tr}}
\newcommand{\tr}{\operatorname{tr}}

\newcommand{\C}{\mathcal{C}}
\newcommand{\CC}{\mathbb{C}}
\newcommand{\RR}{\mathbb{R}}
\newcommand{\Z}{\mathbb{Z}}
\newcommand{\F}{\mathcal{F}}
\newcommand{\gr}{\operatorname{gr}}
\newcommand{\ind}{\operatorname{Ind}}
\newcommand{\OP}{\mathcal{OP}}
\newcommand{\EP}{\mathcal{EP}}
\newcommand{\SOP}{\mathcal{SOP}}
\newcommand{\wtd}{\widetilde}
\newcommand{\td}{\tilde}
\newcommand{\ZZ}{\mathbb{Z}}
\newcommand{\HH}{\operatorname{HH}}
\newcommand{\gl}{\mathfrak{gl}}
\newcommand{\SL}{\mathfrak{sl}}

\newtheorem{theorem}{Theorem}
\newtheorem{proposition}[theorem]{Proposition}
\newtheorem{lemma}[theorem]{Lemma}
\newtheorem{remark}[theorem]{Remark}

\numberwithin{theorem}{section}

\newcommand\numberthis{\addtocounter{equation}{1}\tag{\theequation}}
\theoremstyle{definition}
\newtheorem*{exmp}{Example}

\usepackage{hyperref}

\begin{document}
\title{Cocenters of  Hecke-Clifford and Spin Hecke Algebras}
\author{Michael Reeks}
\address{Department of Math., University of Virginia,
Charlottesville, VA 22904} \email{
mar3nf@virginia.edu}
\begin{abstract} We determine a basis of the cocenter (i.e., the trace or zeroth Hochschild homology) of the degenerate affine Hecke-Clifford and spin Hecke algebras in classical types.
\end{abstract}

\maketitle

\section{Introduction}

The degenerate affine Hecke-Clifford algebra was introduced in type $A_{n-1}$ in \cite{Naz}, and in all classical types in \cite{KW}. These algebras are variations on the degenerate (or graded) affine Hecke algebras, which were introduced independently in \cite{Dr} (to study Yangians) and in \cite{Lu} (to study representations of reductive $p$-adic groups). The degenerate affine spin Hecke algebras were introduced in type $A_{n-1}$ in \cite{W}, and in all classical types in \cite{KW}. These are degenerate affine Hecke algebras associated to the spin Weyl groups. Hecke-Clifford algebras and spin Hecke algebras are closely related to the study of the spin representation theory of classical Weyl groups \cite{Joz}.

 In studying the representation theory of these algebras, it is useful to have a description of the cocenter, or trace: the quotient of the algebra by the linear subspace spanned by the commutators. This space is also the zeroth Hochschild homology of the algebra. In this paper, we determine a linear basis for the cocenter of the degenerate affine Hecke-Clifford and spin Hecke algebras in the classical types, adapting methods used in \cite{CiHe} to solve the corresponding problem for degenerate affine Hecke algebras. The Hochschild homology of the degenerate affine Hecke algebras was computed first in \cite{So}, which also provides a description of the cocenter.

Let $\aHC_X$ be the degenerate affine Hecke-Clifford algebra associated to the Weyl group $W=W_X$ of type $X$, for $X= A_{n-1}$, $B_n$, or $D_n$, as constructed in \cite{Naz} (for type $A$) and in \cite{KW} (for types $B$ and $D$). As a vector space, $\aHC_X$ is isomorphic to $S(V) \otimes \C_V \otimes \CC W$, where $S(V)$ is the symmetric algebra of the reflection representation of $W$ and $\C_V$ is the Clifford algebra associated to V.

 For certain conjugacy classes of $W$, we associate a subset $J_C$ of the root system and pick an element $w_C \in C\cap W_{J_C}$, where $W_{J_C}$ is the parabolic subgroup associated to $J_C$. We then fix a basis $f_{J_C;i}$ of a certain subspace of $S(V^2)$, the subspace of the symmetric algebra spanned by the squares of generators, which is determined by the action of $W_{J_C}$. The first main result of the paper is that the set $\{w_C f_{J_C; i}\}$, as $C$ runs over the distinguished conjugacy classes for each type described in Sections 3.3 and 3.4, forms a basis for the cocenter $\coaHC_X = (\aHC_X/[\aHC_X,\aHC_X])_{\overline{0}}$. 

The proof that $\{w_C f_{J_C; i}\}$ is a spanning set for $\coaHC_X$ relies on several reduction results. We show that an arbitrary element of $\aHC_X$ can be reduced mod $[\aHC_X, \aHC_X]$ to an element containing no instances of generators of the Clifford algebra. We then show that Weyl group elements belonging to certain conjugacy classes, which vary between types, vanish in the cocenter. Then we take advantage of a filtration of $\aHC_X$ to pass to the associated graded object, and use methods developed in \cite{CiHe} to prove that  $\{w_C f_{J_C; i}\}$ spans in that setting. Finally, we can lift the the spanning set to the ungraded object. 

To prove linear independence, we establish a trace formula for parabolically induced $\aHC_X$-modules . This trace formula allows us to separate the $w_C$'s by their action on subspaces of $S(V^2)$. By applying this trace formula to the action of the $w_C's$ on a set of irreducible modules of the parabolic subalgebras, we obtain the linear independence result and the first main theorem of the paper, Theorem \ref{thm:linind}.  

Next, let $\saH_X$ be the degenerate affine spin Hecke algebra associated to the spin Weyl group $W^-$ of type $X$, for $X= A_{n-1},\ B_n,$ or $D_n$. These algebras were originally constructed in \cite{W} (in type $A$) and in \cite{KW} (in types $B$ and $D$). As a vector space, $\saH_X$ is isomorphic to $\CC W^- \otimes \C\langle b_1, \ldots, b_n\rangle$, where $\C\langle b_1, \ldots, b_n\rangle$ is the skew polynomial algebra. The algebra $\saH_X$ is Morita superequivalent to $\aHC_X$ in the sense that $\saH_X \otimes \C_V \cong \aHC_X$.

 The second main result of the paper is that the set $\{t_C f_{J_C;i}^-\}$, as $C$ runs over the same distinguished conjugacy classes of $W$ as in the Hecke-Clifford case, forms a basis of $\overline{\saH_X}$. We prove that this forms a spanning set by following similar reduction procedures as in the Hecke-Clifford case; our approach here does not rely on the Morita superequivalence. Finally, we use the Morita superequivalence to carry the linear independence of the basis for $\coaHC_X$ over to the spanning set of $\overline{\saH_X}$.

The paper is organized as follows. In section 2, we establish notations and describe the degenerate affine Hecke-Clifford algebras in types $A$, $B$, and $D$. In section 3, we prove a series of lemmas to reduce an arbitrary element in the cocenter of $\aHC_X$ to a corresponding element with no Clifford algebra generators, and then prove that Weyl group elements not belonging to certain distinguished conjugacy classes vanish in the cocenter. In section 4, we establish a spanning set of the associated graded object $\coGRaHC_X$ and lift it to $\coaHC_X$ in each type. We then proceed in section 5 to prove that these spanning sets are linearly independent by establishing a trace formula for parabolically induced module. In section 6, we construct the degenerate spin affine Hecke algebra in each type. Section 7 contains reduction formulas similar to those in section 3, with proofs adapted to the new setting. Finally, in section 8, we establish a spanning set for $\cosaH_X$ and then take advantage of the Morita superequivalence to prove that it is linearly independent.

\subsection*{Acknowledgements}
The author would like to thank Weiqiang Wang for providing guidance and valuable discussions about the topics of the paper. 
\section{Preliminaries on Hecke-Clifford algebras}

We establish basic notations and definitions, and then recall the definition of the degenerate affine Hecke-Clifford algebra in types $A$, $B$, and $D$. We then recall some basic facts about these algebras, including a PBW property and a filtration, and finally define the cocenter.

\subsection{Root systems and the Weyl group}

Let $\Phi = (V_0, R, V_0^\vee, R^\vee)$ be a semisimple real root system: $V_0$ and $V_0^\vee$ are finite dimensional real vector spaces, $R$ and $R^\vee$ span $V_0$ and $V_0^\vee$ respectively and, there is a bijection $R \leftrightarrow R^\vee$ such that $(\alpha, \alpha^\vee) = 2$, and $R$ and $R^\vee$ are preserved by the reflections $s_\alpha: v\mapsto (v- (v,\alpha^\vee)\alpha)$. Set 
\[
    V = \CC \otimes_\RR V_0 \text{              and                 } V^\vee = \CC \otimes_\RR V_0^\vee.
\] 
Let $W$ be the finite Weyl group of $\Phi$, the subgroup of $GL(V)$ generated by $s_\alpha$, $\alpha \in R$. Fix a choice of positive roots $R^+$ and positive coroots $(R^+)^\vee$, and let $\Pi = \{\alpha_1, \ldots, \alpha_r\} \subset R^+$ be a basis, the set of simple roots. Then $W$ is a finite Coxeter group with presentation
\begin{eqnarray} \label{eq:weyl}
\langle s_1,\ldots,s_n | (s_is_j)^{m_{ij}} = 1,\ m_{i i} = 1,
 \ m_{i j} = m_{j i} \in \Z_{\geq 2}, \text{for }  i
 \neq j \rangle
\end{eqnarray}
where $m_{ij} \in \{1,2,3,4,6\}$ is specified by the Coxeter-Dynkin diagrams, wherein the vertices correspond to generators of $W$. Two generators $s_i$ and $s_j$, $i\not=j$, have $m_{ij} = 2$ if there is no edge between $i$ and $j$, $m_{ij}=3$ if $i$ and $j$ are connected by an unmarked edge, and $m_{ij} = \ell$ if the edge connecting $i$ and $j$ is labeled with an $\ell \geq 4$. 
 \begin{equation*}\numberthis \label{coxdynk diagrams}
 \begin{picture}(150,45) 
 \put(-99,18){$A_{n}$}
 \put(-30,20){$\circ$}
 \put(-23,23){\line(1,0){32}}
 \put(10,20){$\circ$}
 \put(17,23){\line(1,0){23}}
 \put(41,22){ \dots }
 \put(64,23){\line(1,0){18}}
 \put(82,20){$\circ$}
 \put(89,23){\line(1,0){32}}
 \put(122,20){$\circ$}
 \put(-30,9){$1$}
 \put(10,9){$2$}
 \put(74,9){${n-1}$}
 \put(122,9){${n}$}
 \end{picture}
 \end{equation*}
 %
 \begin{equation*}
 \begin{picture}(150,55) 
 \put(-99,18){$B_{n}(n\ge 2)$}
 \put(-30,20){$\circ$}
 \put(-23,23){\line(1,0){32}}
 \put(10,20){$\circ$}
 \put(17,23){\line(1,0){23}}
 \put(41,22){ \dots }
 \put(64,23){\line(1,0){18}}
 \put(82,20){$\circ$}
 \put(89,23){\line(1,0){32}}
 \put(122,20){$\circ$}
 \put(-30,10){$1$}
 \put(10,10){$2$}
 \put(74,10){${n-1}$}
 \put(122,10){${n}$}
 %
 \put(102,24){$4$}
 \end{picture}
 \end{equation*}
%
%
 \begin{equation*}
 \begin{picture}(150,75) 
 \put(-99,28){$D_{n} (n \ge 4)$}
 \put(-30,30){$\circ$}
 \put(-23,33){\line(1,0){32}}
 \put(10,30){$\circ$}
 \put(17,33){\line(1,0){15}}
 \put(35,30){$\cdots$}
 \put(52,33){\line(1,0){15}}
 \put(68,30){$\circ$ }
 \put(75,33){\line(1,0){32}}
 \put(108,30){$\circ$}
 \put(113,36){\line(1,1){25}}
 \put(138,61){$\circ$}
 \put(113,29){\line(1,-1){25}}
 \put(138,-1){$\circ$}
 \put(-29,20){$1$}
 \put(10,20){$2$}
 \put(60,20){$n-3$}
 \put(117,30){$n-2$}
 \put(145,0){$n-1$}
 \put(145,60){$n$}
 %
 \end{picture}
 \end{equation*}

For every subset $J\subset \Pi$, denote by $W_J$ the parabolic subgroup of $W$, generated by $\{s_i = s_{\alpha_i} | \alpha \in J\}$. Denote by $V_J, R_J, V_J^\vee, R_J^\vee$ the corresponding vector spaces. 

\subsection{The Clifford algebra}

The reflection representation $V$ carries a $W$-invariant nondegenerate bilinear form $(-,-)$, which gives rise to an identification $V^*\cong V$. We identify $V^*$ with a suitable subspace of $\CC^n$ and choose a standard orthonormal basis $\{e_i\}$ of $\CC^n$. 

Denote by $\C_n$ the Clifford algebra assosciated to $(\CC^n, (-,- ))$. It is an associative $\CC$-algebra with identity which contains $\CC^n$ as a subspace and is generated by elements of $\CC^n$ subject to the relation

\begin{eqnarray}\label{eq:clifford}
uv + vu= (u,v) \qquad u,v\in \CC^n.
\end{eqnarray}

Set $c_i = \sqrt{2} e_i$ for each $i$. Then $\C_V$ is generated by elements $c_1, \ldots, c_n$ subject to relations 
\begin{eqnarray}\label{eq:cliffordWeyl}
c_i^2=1, \quad c_ic_j = -c_jc_i \quad i\not=j.
\end{eqnarray}

Let $\C_V$ be the Clifford algebra associated to $(V,(-,-))$, which is a subalgebra of $\C_n$. The algebra $\C_V$ has generators $\beta_i$ corresponding to the simple roots $\alpha_i$ of the Lie algebra corresponding to $W$; note that, in this paper, we always choose to work with the Lie algebra $\gl_n$ in type $A_{n-1}$, rather than $\SL_n$. Note that $\C_V$ is naturally a superalgebra with each $\beta_i$ odd. The explicit generators are given in the following table for types $A_{n-1}$, $B_n$, and $D_n$:

 \begin{center}
\begin{tabular}
[t]{|l|l|p{3.5in}|}\hline Type of $W$&$N$ & Generators for
$\mathcal{C}_W$\\
\hline $A_{n-1}$ & $n$&
$\beta_{i}=\frac{1}{\sqrt{2}}(c_{i}-c_{i+1}),1\leq i\leq n-1$\\
\hline $B_{n}$ &$n$&
$\beta_{i}=\frac{1}{\sqrt{2}}(c_{i}-c_{i+1}),1\leq
i\leq n-1$, $\beta_{n}=c_{n}$\\
\hline $D_{n}$ &$n$ &
$\beta_{i}=\frac{1}{\sqrt{2}}(c_{i}-c_{i+1}),1\leq
i\leq n-1$, $\beta_{n}=\frac{1}{\sqrt{2}}(c_{n-1}+c_{n})$\\
\hline
\end{tabular}
\end{center}

The action of $W$ on $V$ preserves the bilinear form $(\ ,\ )$, so $W$ acts on $\C_V$ by automorphisms. This allows us to form the semidirect product $\C_V \rtimes \CC W$, which is also naturally a superalgebra with $\CC W$ even.

\subsection{The degenerate affine Hecke-Clifford algebras}

We recall the degenerate affine Hecke-Clifford algebras of types $A_{n-1}$, $B_n$, and $D_n$, following the descriptions of \cite{Naz} in type $A_{n-1}$ and \cite{KW} in types $B_n$ and $D_n$. Let $S(V)$ be the symmetric algebra of $V$. Then $S(V) \cong \CC[x_1, \ldots, x_n]$, where $\{x_1, \ldots, x_n\}$ is a basis of $V^*$. Let $u\in \CC$ and set $W= S_n$, the Weyl group of type $A_{n-1}$.

 The degenerate affine Hecke Clifford algebra of type $A_{n-1}$, $\aHC_{A_{n-1}}$, is the $\CC$-algebra generated by $x_1, \ldots, x_n$, $c_1, \ldots, c_n$, and $S_n$, subject to relations making $\CC[x_1, \ldots, x_n]$, $\C_V$, and $\CC S_n$ subalgebras, along with the additional relations: 
\begin{align}
x_i c_i &= -c_i x_i, \quad x_i c_j = c_j x_i \quad (i\not= j) \label{xici},\\
\sigma c_i & = c_{\sigma(i)} \sigma \quad (1\leq i \leq n, \sigma \in S_n)\label{sigmaci},\\
x_{i+1}s_{i}-s_{i}x_{i} &  =u(1-c_{i+1}c_{i})    \label{xisi}, \\
x_{j}s_{i} &  =s_{i}x_{j}\quad (j\neq i,i+1)  \label{xjsi}.
\end{align}

Denote the action of $S_n$ on $S(V)$ by $f\mapsto f^\sigma$, $f\in S(V)$, $\sigma \in S_n$.

Next, let $W= W_{D_n}$, the Weyl group of type $D_n$. It is generated by elements $s_1, \ldots, s_n$ where $s_1, \ldots, s_{n-1}$ are subject to the defining relations of $S_n$, and there are the additional relations:
\begin{align}
s_{i}s_{n} &  =s_{n}s_{i}\quad (i\neq n-2) ,\\
s_{n-2}s_{n}s_{n-2} &  =s_{n}s_{n-2}s_{n}, \quad s_n^2=1.
\label{eq:braidD}
\end{align}

The degenerate affine Hecke-Clifford algebra of type $D_n$, $\aHC_{D_n}$, is generated by $x_i, c_i, s_i$, $1\leq i \leq n$, subject to relations making $\CC[x_1, \ldots, x_n]$, $\C_V$, and $\CC W$ subalgebras, along with the relations (\ref{xici}) -- (\ref{xjsi}) and the additional relations:
\begin{align}
s_{n}c_{n} &  =-c_{n-1}s_{n} \nonumber,  \\
s_{n}c_{i} & =c_{i}s_{n}  \quad (i\neq n-1,n) \nonumber,  \\
s_{n}x_{n} + x_{n-1}s_{n}&  = -u(1+c_{n-1}c_{n}) \label{Dsnxn} ,\\
s_n x_i & = x_i s_n   \quad (i\neq n-1,n).   \nonumber
\end{align}

Finally, let $W= W_{B_n}$, the Weyl group of type $B_n$. It is generated by elements $s_1, \ldots, s_n$, where $s_1, \ldots, s_{n-1}$ are subject to the defining relations on $S_n$, and there are the additional relations: 
\begin{align} s_is_n &= s_ns_i \quad(1\leq i \leq n-2) \\
(s_{n-1} s_n)^4 &= 1, \quad s_n^2 =1.
\end{align} The simple reflections of $W$ lie in two different conjugacy classes: $s_n$ is not conjugate to $s_1, \ldots, s_{n-1}$. 

Let $u,v \in \CC$. The degenerate affine Hecke-Clifford algebra of type $B_n$, $\aHC_{B_n}$, is generated by $x_i, c_i, s_i$, $1\leq i \leq n$, subject to relations subject to relations making $\CC[x_1, \ldots, x_n]$, $\C_V$, and $\CC W$ subalgebras, along with the relations \ref{xici} -- \ref{xjsi},   and the additional relations:
\begin{align*}
s_{n}c_{n} &  =-c_{n}s_{n}, \\
s_{n}c_{i} &=c_{i}s_{n}\quad (i\neq n) ,\\
s_{n}x_{n} +x_{n}s_{n}&  = -\sqrt{2} \,v,\\
s_n x_i &=x_i s_n \quad  (i\neq n).
\end{align*}

The PBW theorems for the degenerate affine Hecke-Clifford algebras in type $A_{n-1}$ were proved in \cite{Naz} and in \cite{KW} using different methods, and in types $B_n$ and $D_n$ in \cite{KW}. The even center of these algebras- the subalgebra of even central elements - was also established in \cite{KW}.
\begin{proposition}  \label{PBW:DB}
Let $X= A_{n-1}$, $D_n$ or $B_n$. 
\begin{enumerate} \item The multiplication of subalgebras
$\CC[x_1, \ldots, x_n], \C_V$, and $\CC W$ induces a vector space isomorphism%
\[
\CC[x_1, \ldots, x_n]\otimes \C_V \otimes \CC W\longrightarrow\aHC_X.
\]
Equivalently, the elements $\{x^\alpha c^\epsilon w| \alpha\in\Z_+^n,
\epsilon \in\Z_2^n, w\in W\}$ form a linear basis for $\aHC_X$.
\item Let $X={A_{n-1}}, {B_n}$ or $D_n$. Then  $$Z(\aHC_X)_{\overline{0}}\cong \CC [x_1^2,\ldots,x_n^2]^{W_X}.$$
\end{enumerate}
\end{proposition}

Each of these algebras is naturally a superalgebra with even generators from $S(V)$ and $\CC W$ and odd generators from $\C_V$.

Denote by $S(V^2)$ the subspace of $S(V)$ spanned by the squares of the basis elements in $S(V)$. Thus $Z(\aHC_X)_{\overline{0}} \cong S(V^2)^{W_X}$. 

\subsection{A filtration of $\aHC_X$}

In any of the algebras defined in Section 2.3, we can define a notion of degree as follows. From the various PBW basis theorems, we see that every $h\in \aHC_X$, for $X= A_{n-1}, B_n$, or $D_n$, can be written $$ h = \sum_{w\in W} a_w c_w w$$ where $a_w\in S(V)$ and $c_w \in \C_V$. Set $$|h | = \max_{w\in W} \{|a_w|\}$$ where $|a_w|$ denotes degree in $S(V)$. Set $\F^j \aHC_X = \{h\in \aHC_X |\  |h|\leq j\}$; then we have a filtration $$\C_V \rtimes \CC W = \F^0 \aHC_X \subset \F^1 \aHC_X \subset \ldots$$ Let $\gr(\aHC_X)$ be the associated graded algebra. It is clear from the defining relations for $\aHC_X$ that $\gr(\aHC_X) \cong \GRaHC_X$, the degenerate affine Hecke-Clifford algebra with parameter $u=0$. 

\subsection{Parabolic Subalgebras}

For any $J\subset \Pi$, define the parabolic subalgebra $\aHC_{X,J}$ to be the subalgebra of $\aHC_X$ generated by $W_J$, $\C_V$, and $\CC[x_1, \ldots, x_n]$. For every $\aHC_{X,J}$-module $M$, define the parabolically induced module $$\ind_{\aHC_{X,J}}^{\aHC_X} M := \aHC_X \otimes_{\aHC_{X,J}} M.$$
\subsection{The cocenter}

For any $h,h'\in \aHC_X$, define the commutator $[h,h'] = hh' - h' h$. Let $[\aHC_X, \aHC_X]$ be the submodule of $\aHC_X$ generated by all commutators. The \emph{cocenter} of $\aHC_X$ is the space $$\coaHC_X := \left(\frac{\aHC_X}{[\aHC_X, \aHC_X]}\right)_{\overline{0}}.$$ The main goal of the paper is to find a linear basis for the cocenter. 

Note that we restrict our definition to only the \emph{even} cocenter. Referring to the example of the cocenter of the finite Hecke-Clifford algebra, as studied in \cite[Section 4.1]{WW}, gives intuition as to why this is the correct notion of cocenter.

  Wan and Wang study the space of trace functions on the finite Hecke-Clifford algebra $\mathcal{H}_n$: linear functions $\phi: \mathcal{H}_n \rightarrow \CC$ such that $\phi([h,h'])=0$ for all $h,h'\in \mathcal{H}_n$, and $\phi(h)=0$ for all $h\in (\mathcal{H}_n)_{\overline{1}}$. This latter requirement encodes the information that odd elements act with zero trace on any $\ZZ_2$-graded $\mathcal{H}_n$-module (because multiplication by an odd element results in a shift in degree). The space of such trace functions is clearly canonically isomorphic to the dual of the even cocenter, rather than of the full cocenter. Moreover, since the even cocenter of $\mathcal{H}_n$ has dimension equal to the number of irreducible $\ZZ_2$-graded representations of $\mathcal{H}_n$, this restriction sets up the desired linear isomorphism between the space of trace functions and the linear span of the irreducible representations (the matrix of this isomorphism is the character table of the algebra).

  In the affine case, we see that the trace of the action of an odd element on any $\aHC_X$-module is still zero, due to the same degree shift. Hence we deduce that the interesting information about traces of $\aHC_X$ (and, thus, much of the interesting representation-theoretic information about $\aHC_X$) is contained in the even cocenter.

\section{Reduction}

The goal of this section is to show that an element $h= \sum_{w\in W} a_w c_w w \in \aHC_{X}$ is congruent in the cocenter to a (possibly differently indexed) linear combination $h= \sum_i a_i w_i$ without any Clifford algebra elements, and to show that certain conjugacy classes of Weyl group elements vanish in the cocenter. 

\subsection{Clifford reduction in type $A_{n-1}$}

We adapt the procedure in \cite{WW}, where similar formulas are developed in the finite and non-degenerate case, with appropriate modifications. Let $w_{(n)} = s_1s_2\ldots s_{n-1} = (1 \ 2 \ \ldots \ n)$. The following follows directly from the defining relations in $\aHC_{A_{n-1}}$.
\begin{lemma}\label{wncommuting} In $\aHC_{A_{n-1}}$, we have
\begin{align*} w_{(n)} c_i = c_{i+1} w_{(n)}& \qquad \text{for } 1\leq i\leq n-1,\\
 w_{(n)} c_n = c_1 w_{(n)}.\ &\end{align*}
\end{lemma}

For $n\in \ZZ^{>0}$, let $[n] = \{1,2,\ldots, n\}$. For any subset $I\subseteq [n]$, let $c_I = \Pi_{i\in I} c_i$. Note that it suffices to consider only elements $w c_I$ where $|I|$ is even, since $|c_i| = 1$ for all $i$ and we are studying the even cocenter. 
\begin{lemma}\label{wn even subsets} For $I\subseteq [n]$ with $|I|$ even, we have $$w_{(n)} c_I \equiv \pm w_{(n)} \quad \mod [\aHC_{A_{n-1}}, \aHC_{A_{n-1}}].$$\end{lemma}
\begin{proof} Write $I= \{ i_1, \ldots, i_k\}$. Then \begin{align*} w_{(n)} c_I &= (1 \ 2 \ \ldots \ n) c_{i_1} \ldots c_{i_k} \\&= c_{i_1 +1} (1\ 2\ \ldots \ n) c_{i_2} \ldots c_{i_k} \\ &\equiv (1\ 2\ \ldots \ n) c_{i_2} \ldots c_{i_k} c_{i_1 +1} \quad \mod [\aHC_A, \aHC_A] \\
&= \left\{\begin{array}{lr} (-1)^{k-1} w_{(n)} c_{i_1+1} c_{i_2} \ldots c_{i_k} & i_1+1 < i_2\\ (-1)^{k-2} w_{(n)} c_{i_3} \ldots c_{i_k} & i_1+1 = i_2. \end{array} \numberthis \label{wn even subsets eqn}\right.
\end{align*}
Now we have either reduced the size of $I$ by two or increased $i_1$ by one. Since $|I|$ is even, we can continue in this way until no $c_i$ remain. \end{proof}

If $\gamma = (\gamma_1, \ldots, \gamma_k)$ is a sequence of (not necessarily decreasing) positive integers such that $\sum_{i=1}^k \gamma_i = n$, call $\gamma$ a composition of $n$. For such a composition $\gamma$ of $n$, set $w_{\gamma} = w_{\gamma_1} \ldots w_{\gamma_k}$.

\begin{lemma}\label{lm:clifred} Let $\gamma = (\gamma_1, \gamma_2)$ be a composition of $n$ with $\gamma_1, \gamma_2 >0$. Let $I_1 = \{i_1, \ldots, i_a\} \subseteq \{1, \ldots, \gamma_1\}$ and $I_2 = \{j_1,\ldots, j_b\} \subseteq \{\gamma_1+1, \ldots, \gamma_2\}$, and assume that $a+b$ is even. Then we have $$w_\gamma c_{I_1} c_{I_2} \equiv \left\{\begin{array}{lr} 0 & a,b \text{ odd}\\ \pm w_{\gamma} & a,b \text{ even} \end{array} \right. \quad \mod [\aHC_{A_{n-1}},\aHC_{A_{n-1}}].$$ \end{lemma}
\begin{proof}Note that $a$ and $b$ must have the same parity if their sum is even. We have $w_\gamma = w_{\gamma_1} w_{\gamma_2}  = w_{\gamma_2} w_{\gamma_1}$. Suppose that $a$ and $b$ are both odd. Then \begin{align*} w_\gamma c_{I_1} c_{I_2} &= w_{\gamma_1} c_{I_1} w_{\gamma_2} c_{I_2} \\&\equiv w_{\gamma_2} c_{I_2} w_{\gamma_1} c_{I_1}  \mod [\aHC_{A_{n-1}},\aHC_{A_{n-1}}] \\&= w_\gamma c_{I_2} c_{I_1} \\ &= - w_\gamma c_{I_1} c_{I_2}. \end{align*} since commuting $c_{I_1}$ past $c_{I_2}$ yields a sign of $(-1)^{ab}$. Hence $w_\gamma c_{I_1} c_{I_2} \equiv 0 \mod [\aHC_A,\aHC_A]$. 

Next, suppose $a$ and $b$ are even. Note that $\gamma_1 +1 \leq j_1 \leq n-1$, so $c_{j_1}$ anticommutes with all $c_{i_s}$. We have \begin{align*} w_{\gamma} c_{I_1} c_{I_2} &= w_{\gamma_1} w_{\gamma_2} c_{i_1} \ldots c_{i_a} c_{j_1} \ldots c_{j_b} \\ &=- w_{\gamma_1} c_{i_1} w_{\gamma_2} c_{j_1} c_{i_2} \ldots c_{i_a} c_{j_2} \ldots c_{j_b} \\&= -c_{i_1 +1} c_{j_1 +1 } w_\gamma c_{i_2} \ldots c_{i_a} c_{j_2} \ldots c_{j_b} \\ &\equiv -w_\gamma c_{i_2} \ldots c_{i_a} c_{j_2} \ldots c_{j_b}c_{i_1 +1} c_{j_1 +1 }\mod [\aHC_A, \aHC_A]. \end{align*} Now, commuting $c_{i_1+1}$ and $c_{j_1+1}$ has four possible results, depending on which of the two (if either) cancels with the second Clifford element in their subset. In particular, we have $$w_\gamma c_{I_1} c_{I_2} \equiv \left\{\begin{array}{lr} w_\gamma c_{i_1+1} c_{i_2} \ldots c_{i_a} c_{j_1+1} c_{j_2} \ldots c_{j_b} & i_1+1<i_2, \ j_1+1<j_2 \\ -w_\gamma c_{i_1+1} c_{i_2} \ldots c_{i_a} c_{j_3}  \ldots c_{j_b} & i_1+1<i_2, \ j_1+1=j_2 \\ -w_\gamma c_{i_3}  \ldots c_{i_a} c_{j_1+1} c_{j_2} \ldots c_{j_b} & i_1+1=i_2, \ j_1+1<j_2 \\ w_\gamma c_{i_3} \ldots c_{i_a} c_{j_3} \ldots c_{j_b} & i_1+1=i_2, \ j_1+1=j_2. \end{array}\right. $$ In any case, we have either reduced the length of $c_{I_1}$ or $c_{I_2}$ or increased the index of the first element. Continuing in this manner gives the result. \end{proof}

By induction, we have the following:
\begin{proposition}\label{clifred} If $\gamma = (\gamma_1, \ldots, \gamma_k)$ is a composition of $n$, $I\subset [n]$ is an even subset, and $I_k = I\cap \{\sum_{i=1}^{k-1}\gamma_{i} +1, \ldots, \sum_{i=1}^k \gamma_i\}$, then we have \begin{eqnarray}\label{noclifford} w_\gamma c_I \equiv \left\{\begin{array}{lr} \pm w_\gamma & \text{ if every } |I_k| \text{ is even} \\ 0 & \text{else} \end{array}\right. \mod \ [\aHC_{A_{n-1}}, \aHC_{A_{n-1}}].\end{eqnarray}\end{proposition}

 The sign is determined by the structure of each subset. Finally, specializing $\gamma$ to a partition of $n$, we obtain the desired result. 

\subsection{Clifford reduction in types $B_n$ and $D_n$}
We can extend Proposition \ref{clifred} to types $B$ and $D$. The commutation relations between elements of $W$ and elements of $\mathcal{C}_n$ in types $B_n$ and $D_n$ differs from that in type $A_{n-1}$ only in that we have the extra relations $s_n c_n = - c_n s_n$ and $s_n c_n =  - c_{n-1} s_n$, respectively. Let $w_{(n)} = s_1 \ldots s_n$. We have the following versions of Lemmas \ref{wncommuting} and \ref{wn even subsets}: 

\begin{lemma}  In $\aHC_{B_n}$ and $\aHC_{D_n}$, we have
\begin{align*} w_{(n)} c_i = c_{i+1} w_{(n)} &\qquad \text{for } 1\leq i\leq n-1, \\
 w_{(n)} c_n = -c_1 w_{(n)}.\ &\end{align*}

\end{lemma}

\begin{lemma}  Let $X= B_n$ or $D_n$. For $I\subseteq [n]$ with $|I|$ even, we have $$w_{(n)} c_I \equiv \pm w_{(n)} \quad \mod [\aHC_{X}, \aHC_{X}].$$\end{lemma}

The proofs are identical, with an additional $(-1)$ added in equation (\ref{wn even subsets eqn}) if $n\in I$. We also have

\begin{lemma} Let $\gamma = (\gamma_1, \gamma_2)$ be a composition of $n$ with $\gamma_1, \gamma_2 >0$. Let $I_1 = \{i_1, \ldots, i_a\} \subseteq \{1, \ldots, \gamma_1\}$ and $I_2 = \{j_1,\ldots, j_b\} \subseteq \{\gamma_1+1, \ldots, \gamma_2\}$, and assume that $a+b$ is even. Then for $X=B_n$ or $D_n$, we have $$w_\gamma c_{I_1} c_{I_2} \equiv \left\{\begin{array}{lr} 0 & a,b \text{ odd}\\ \pm w_{\gamma} & a,b \text{ even} \end{array} \right. \quad \mod [\aHC_X,\aHC_X].$$ \end{lemma}

There are only two modifications to the proof. In Lemma \ref{lm:clifred}, the only problem occurs if $I_2$ contains both $n-1$ and $n$, so that $c_{I_2}$ ends with $\ldots c_{n-1} c_n$. Then commuting $c_{I_2}$ past $w_{\gamma_2}$ gives in type $D_n$ that $$w_{\gamma_2}\ldots c_{n-1} c_n = \ldots s_{n-1} (c_{n}c_{n-1}) s_n = \ldots (c_{n-1} c_n) w_\gamma = c_{I_2} w_{\gamma_2}$$ Hence there is no impact on the proof. In type $B_n$, there is a sign change which cancels out: we have \begin{align*} w_{\gamma_2} \ldots c_{n-1} c_n & =  \ldots s_{n-1}(-c_{n-1}c_n) s_n \\&= \ldots (-c_n c_{n-1}) w_{\gamma_2} \\&= \ldots (c_{n-1} c_n) w_{\gamma_2}.\end{align*}

Thus, we have the following proposition.

\begin{proposition} Let $X= B_n$ or $D_n$. If $\gamma = (\gamma_1, \ldots, \gamma_k)$ is a composition of $n$, $I\subset [n]$ is an even subset, and $I_k = I\cap \{\sum_{i=1}^{k-1}\gamma_{i} +1, \ldots, \sum_{i=1}^k \gamma_i\}$, then we have \begin{eqnarray}\label{noclifford} w_\gamma c_I \equiv \left\{\begin{array}{lr} \pm w_\gamma & \text{ if every } |I_k| \text{ is even} \\ 0 & \text{else} \end{array}\right.\mod \ [\aHC_{X}, \aHC_{X}].\end{eqnarray} \end{proposition}

\subsection{Conjugacy classes in type $A_{n-1}$}

Though we can apply the reduction formulas from the previous section to remove Clifford algebra generators from our basis elements, they still restrict the Weyl group elements that can appear. 

Let $\OP_n$ be the set of partitions of $n$ with all odd parts. It is proved in \cite{BW} that $\OP_n$ parametrizes the even split conjugacy classes of $\CC W$ in type $A_{n-1}$. These are the even conjugacy classes in $\CC W$ which split into two separate conjugacy classes in the double cover $\CC \widetilde{W}$. It is proved in \cite{Joz2} that the number of even split conjugacy classes is the number of simple $\CC W^-$-modules, so we should expect the combinatorics of these classes to play a role in our bases.

\begin{proposition}\label{conjclass A} If $\lambda$ is a partition of $n$ with $\lambda \notin \OP_n$ and $w \in S_n$ has cycle type $\lambda$, then $w \equiv 0 \mod [\aHC_{A_{n-1}}, \aHC_{A_{n-1}}]$. 
\end{proposition}\begin{proof} Since elements which are conjugate in $\aHC_A$ are congruent in the cocenter, we may take $$w_\lambda = (1 \ldots \lambda_1)(\lambda_1+1 \ldots \lambda_1+\lambda_2)\ldots(\lambda_1+\ldots+\lambda_{n-1}+1 \ldots n).$$ Suppose that $\lambda$ has an even part, and take it without loss of generality to be $\lambda_1$. Then \begin{align*} w_\lambda &\equiv c_1\ldots c_{\lambda_1} w_\lambda c_{\lambda_1}\ldots c_{1}  \mod [\aHC_A, \aHC_A]\\&= c_1 \ldots c_{\lambda_1} c_1 c_{\lambda_1} \ldots c_{2} w_\lambda \\&= (-1)^{\lambda_1-1}w_\lambda. \end{align*} In the last step, we have commuted one of the $c_1$'s past each other Clifford element (a total of $\lambda_1 - 2$ inversions), after which each $c_i$ cancels. Hence, we have $w_\lambda \equiv -w_\lambda  \mod [\aHC_A, \aHC_A]_{\overline{0}}$, whence $w_\lambda \equiv 0  \mod [\aHC_A, \aHC_A]_{\overline{0}}$. \end{proof}

\subsection {Conjugacy classes in types $B_n$ and $D_n$}
Conjugacy classes in the Weyl group in type $B_n$ correspond to bipartitions $(\lambda, \mu)$, $|\lambda| + |\mu| = n$ (cf. \cite{Mac}). For a partition $\lambda$, denote by $\ell(\lambda)$ the number of parts of $\lambda$. Let $\OP$ denote the set of partitions (of any $n$) with all odd parts, and $\EP$ denote the set of partitions with all even parts. The even split conjugacy classes of the spin Weyl group of type $B_n$ are parametrized by bipartitions of $n$ $(\lambda, \mu) \in (\OP, \EP)$, cf. \cite{BW}.

\begin{proposition}\label{conjclass B} Let $(\lambda, \mu)$ be a bipartition of $n$ and $w\in W_{B_n}$ an element in the conjugacy class corresponding to $(\lambda, \mu)$. If $(\lambda, \mu) \notin (\OP, \EP)$, then $w\equiv 0 \mod \ [\aHC_{B_n}, \aHC_{B_n}]$. \end{proposition} \begin{proof} For a bipartition $(\lambda=(\lambda_1, \ldots, \lambda_r), \mu=(\mu_1, \ldots, \mu_s))$, let $$w_{\lambda, \mu} = (1, \ldots, \lambda_1) \ldots (\sum_{j=1}^{r-1}\lambda_{j}+1, \ldots, |\lambda|)(|\lambda|+1, \ldots, |\lambda|+\mu_1) \ldots (|\lambda|+\sum_{j=1}^{s-1}\mu_{j}+1, \ldots, n),$$ where the $\lambda$-cycles are understood to be positive, and the $\mu$-cycles negative.  We claim that unless $(\lambda, \mu) \in (\mathcal{OP}, \mathcal{EP})$ with $\ell(\mu)$ even, $w_{\lambda, \mu} \equiv 0 \mod [\aHC_{B_n}, \aHC_{B_n}]$. Indeed, if $\lambda$ has even part $\lambda_i$, let $c=c_{\lambda_i +1} c_{\lambda_i +2} \ldots c_{\lambda_{i+1}}$. Then, as in type $A$, $$c w_{\lambda, \mu} c^{-1} = -w_{\lambda, \mu}.$$ If $\mu$ has an odd part $\mu_i$, we may assume without loss of generality that it corresponds to a cycle containing $n$, adjusting $w_{\lambda, \mu}$ if necessary. Let $c= c_{\mu_i +1}\ldots c_{n}$ (the length of $c$ is $\mu_i-1$). Then $c w_{\lambda, \mu} c^{-1} = (-1)^{\mu_i-1} cc^{-1} w_{\lambda,\mu} = -w_{\lambda,\mu}$. 

For example, if $(\lambda,\mu) = (\{2\},\{3\})$, $w_{\lambda,\mu} = (12)(345)$. We have \begin{align*} c_1 c_2 (12)(345) c_2 c_1 &= c_1 c_2 c_1 c_2 (12)(345) \\&= - c_1^2 c_2^2(12)(345) \\&= -(12)(345).\end{align*} Also, \begin{align*} c_3c_4c_5 (12)(345) c_5c_4c_3 &= (-1) c_3 c_4 c_5 c_3 c_5 c_4 (12)(345)\\ &= (-1)^3 c_3^2 c_4^2 c_5^2 (12)(345) \\&= -(12)(345). \end{align*}

Finally, if $\ell(\mu)$ is odd, conjugating by $c_n$ yields $w_{\lambda,\mu} \equiv - w_{\lambda,\mu}$. \end{proof}

Conjugacy classes in the Weyl group in type $D_n$ also correspond to bipartitions. Let $\SOP$ denote the set of partitions (of any $n$) with distinct odd parts; the set $\SOP$ parametrizes the even split conjugacy classes of the spin Weyl group of type $D_n$, cf. \cite{BW}.

\begin{proposition}\label{conjclass D} Let $(\lambda, \mu)$ be a bipartition of $n$ and $w\in W_{D_n}$ an element in the conjugacy class corresponding to $(\lambda, \mu)$. If $n$ is odd and $(\lambda, \mu) \notin (\OP, \EP)$ with $\ell(\mu)$ even, then $w\equiv 0 \mod \ [\aHC_{D_n}, \aHC_{D_n}]$. If $n$ is even and $(\lambda, \mu) \notin (\EP, \OP)$ with $\ell(\mu)$ even and $(\lambda, \mu) \notin (\emptyset, \SOP)$, then $w \equiv 0 \mod [\aHC_{D_n}, \aHC_{D_n}]$. \end{proposition} \begin{proof}  For a bipartition $(\lambda, \mu)$, let $w_{\lambda, \mu}$ be as above. If $n$ is odd, we have that $w_{\lambda, \mu} \equiv 0$ mod $[\aHC_{D_n}, \aHC_{D_n}]_{\overline{0}}$ unless $(\lambda, \mu) \in (\mathcal{OP}, \mathcal{EP})$ with $\ell(\mu)$ even by the same arguments as in type $B$. If $n$ is even and $(\lambda, \mu) \in (\emptyset, \mathcal{SOP})$, conjugation by $c$ as above does not fix $w_{\lambda, \mu}$ up to sign, so this case does not vanish. \end{proof}

\section{Spanning sets of the cocenter in $\aHC_X$}

We pass to the associated graded object of the Hecke-Clifford algebra, which is isomorphic to the Hecke-Clifford algebra with parameter identically 0. We establish a spanning set of the cocenter in this case, and then lift it to $\coaHC_X$ using an algebraic argument.

\subsection{Spanning set of $\overline{\GRaHC_X}$}
Let $C$ be a conjugacy class of $W_J$, $J\subseteq I$. We say that $C$ is elliptic in $W_J$ if $W_{J'} \cap C = \emptyset$ for all proper subsets $J'\subset J$.  Any element of $W_J$ which is a member of an elliptic conjugacy class is called an elliptic element in $W_J$. 

\begin{exmp} In $W_{A_{n-1}} = S_n$, there is a unique elliptic conjugacy class, which corresponds to the partition $(n)$. The elliptic elements are the $n$-cycles. For any connected subset $J$ of the root system, the elliptic elements in $W_J$ are the $(|J|+1)$-cycles.

In $W_{B_n}$ and $W_{D_n}$, the unique elliptic conjugacy class corresponds to the bipartition $(\emptyset, (n))$, and the elliptic elements are the negative $n$-cycles.

\end{exmp}

We say that two subsets $J_1, J_2 \subset I$ are $W$-equivalent, $J_1 \sim_W J_2$, if there exists a $w\in W$ such that $w(J_1) = J_2$. Set $\mathcal{I} = 2^I/\sim_W$, the set of equivalence classes of subsets of $I$ for the equivalence relation $\sim_W$. For any conjugacy class $C$ of $W$, set $J_C$ to be the minimal element (with respect to cardinality) of $\mathcal{I}$ such that $C\cap J_C \not= \emptyset$-- since there is exactly one element of $\mathcal{I}$ of each cardinality, such a $J_C$ must exist.  Note that  if $C$ is an elliptic conjugacy class, $J_C = I$.

  Any $w \in C\cap J_C$ is by definition elliptic in $W_{J_C}$. Fix one such elliptic element, $w_C \in W_{J_C}$, and let $.^JW^J$ be a set of minimal length representatives for $W_J/W\backslash W_J$. We have the following result due to \cite{CiHe} linking centralizers of elliptic elements in parabolic subgroups to the normalizers of the parabolic subgroups: 

\begin{proposition}\cite[Proposition 2.4.3]{CiHe} \label{prop:normalizers} \label{centnorm} Let $J\subset I$ and let $w\in W_J$ be an elliptic element. Let $Z= \{z\in^JW^J| z(J) = J\}$. Then we have $$W_J C_W(w) = N_W(W_J) = W_J Z W_J.$$ \end{proposition}

Now we establish spanning sets of the cocenter in each type. 

Recall that $S(V^2)$ is the subspace of $S(V)$ spanned by the squares of basis elements, that $\GRaHC_X$ is the affine Hecke-Clifford algebra of type $X$ with parameter identically 0, that $\gr(\aHC_X) \cong \GRaHC_X$. Thus we have $\GRaHC_X \cong \CC W \ltimes (\mathcal{C}_n \otimes S(V^2))$ as $\CC$-algebras. Hence we certainly have $\overline{\GRaHC_X} \subset \operatorname{span}\{w c^\alpha S(V^2)\}$, where $w\in W$ and $\alpha \in \Z_2^n$. 

\begin{proposition} We have $\overline{\GRaHC_x} = \operatorname{span} \{w S(V^2)\}$.\end{proposition} 
\begin{proof} Apply Proposition \ref{lm:clifred} and the corresponding results for types $B_n$ and $D_n$ to each element in the spanning set. Every element will thus either be congruent to 0 or to an element in $w_C S(V^2)$ for some $C$ in the cocenter.\end{proof}

Next, let $x, y \in W$ and $f\in S(V^2)$. Then $$xyx^{-1} f \equiv yx^{-1} f x = y f^{x^{-1}} \mod [\GRaHC_X, \GRaHC_X]$$ where $f^\sigma$ denotes the action of $\sigma$ on $f$ by conjugation. Hence $\operatorname{span} \{w S(V^2)\} = \{w_C S(V^2)\}$, where $C$ is the conjugacy class of $w$ and $w_C$ is a representative. 

Now we restrict the conjugacy classes of Weyl group elements which may appear. 

\begin{proposition} \label{spanlem}\begin{enumerate} \item We have $\overline{\GRaHC_{A_{n-1}}} = \operatorname{span} \{w_\lambda S(V^2)\}_{\lambda \in \mathcal{OP}_n}$. 
\item We have $\overline{\GRaHC_{B_{n}}} = \operatorname{span} \{w_{\lambda,\mu} S(V^2)\}_{(\lambda,\mu) \in (\mathcal{OP},\mathcal{EP})\text{, }\ell(\mu)\text{ even}}$. 
\item If $n$ is odd, $\overline{\GRaHC_{D_{n}}} = \operatorname{span} \{w_{\lambda,\mu} S(V^2)\}_{(\lambda,\mu) \in (\mathcal{OP},\mathcal{EP})\text{, }\ell(\mu)\text{ even}}$. If $n$ is odd,  $\overline{\GRaHC_{D_{n}}} = \operatorname{span} \{w_{\lambda,\mu} S(V^2)\}$ with ${(\lambda,\mu) \in (\mathcal{OP},\mathcal{EP}})$ or $(\lambda, \mu) \in (\emptyset, \mathcal{SOP})\text{, }\ell(\mu)\text{ even}$. \end{enumerate}\end{proposition}

\begin{proof} By Propositions \ref{conjclass A}, \ref{conjclass B}, and \ref{conjclass D}, every element not in these conjugacy classes is congruent to 0 in the cocenter, so removing them from a set does not change the span. \end{proof}

Finally, we restrict to a subspace of the symmetric algebra.

\begin{proposition} \label{cihe lem} Fix a conjugacy class $C$ of $W$, and let $J= J_C$. Then we have $$w_C S(V^2) \equiv w_C S((V^{2})^{W_J})^{N_W(W_J)} \mod [\GRaHC_X, \GRaHC_X].$$ \end{proposition}

\begin{proof} We follow \cite[Section 6]{CiHe}. We have $V^2= (V^2)^{W_J} \oplus U$ as a $W_J$-module, where $U$ is spanned by $\{x_i^2 | i \in J\}$. Since $w_C$ is elliptic in $W_J$, it acts faithfully on $U$, so $1-w_C$ is invertible on $U$. Let $f\in S(V^2)$ and $u\in U$. Since $1-w_C$ has full rank, there exists a $v\in U$ such that $v- w_C(v) = u$. Thus, we see that $$u w_C f = v w_C f - w_C(v) w_C f = vw_C f - w_C f v = [v, w_C f].$$ Hence $U w_C S(V^2) \in [\GRaHC_X, \GRaHC_X]$. Therefore we have $$w_CS(V^2) = w_C S(U) S((V^2)^{W_J}) = S(U) w_C S((V^2)^{W_J}) \subset w_C S((V^2)^{W_J}) + [\GRaHC_X, \GRaHC_X].$$ Let $f \in S((V^2)^{W_J})$ and $x\in C_W(w_C)$. Then we have \begin{align*}w_C f \equiv x w_C f x^{-1} &= w_C x f x^{-1}\\&= w_C f^x \mod [\GRaHC_X, \GRaHC_X]. \end{align*} Hence we can average over the centralizer of $w_C$ to obtain that $$w_C f = \frac{1}{|C_W(w_C)|} \sum_{x\in C_W(w_C)} w_C f^x \in w_C S((V^2)^{W_J})^{C_W(w_J)}.$$ Finally, apply Proposition \ref{centnorm} to get $$w_C f \in w_C S((V^2)^{W_J})^{N_W(W_J)}$$ using the fact that $S((V^2)^{W_J})^{C_W(w)} = S((V^2)^{W_J})^{W_J C_W(w)}$.
\end{proof}

For each $C$, let $\{f_{{J_C};i}\}$ be a basis of the vector space $S((V^2)^{W_{J_C}})^{N_W(W_{J_C})}$. Propositions \ref{spanlem} and \ref{cihe lem} give us the following.

\begin{proposition}\label{thm:spanning} \begin{enumerate} \item The set $\{w_\lambda f_{J_\lambda;i}\}_{\lambda \in \mathcal{OP}_n}$ spans $\overline{\GRaHC_{A_{n-1}}}$. 

\item The set $\{w_{\lambda,\mu} f_{J_{\lambda,\mu;i}}\}_{(\lambda ,\mu)\in (\mathcal{OP},\mathcal{EP})\text{, }\ell(\mu)\text{ even}}$ spans $\overline{\GRaHC_{B_{n}}}$. 

\item If $n$ is odd, the set $\{w_{\lambda,\mu} f_{J_{\lambda,\mu;i}}\}_{\shortstack[l]{\tiny$(\lambda ,\mu)\in (\mathcal{OP},\mathcal{EP}),$\\\tiny$ \ell(\mu) $ even}}$ spans $\overline{\GRaHC_{D_{n}}}$. \\

\noindent If $n$ is even, the set $\{w_{\lambda,\mu} f_{J_{\lambda,\mu;i}}\}$ with $(\lambda ,\mu)\in (\mathcal{OP},\mathcal{EP})$ or $(\lambda,\mu) \in (\emptyset, \mathcal{SOP}_n)$ with $ \ell(\mu)$ even spans $\overline{\GRaHC_{D_{n}}}$.

\end{enumerate}
\end{proposition}

\subsection{Spanning set of $\coaHC_X$}

The goal of this section is to lift the spanning set constructed above for $\overline{\GRaHC_X}$ to $\coaHC_X$. The proof is motivated by \cite[Section 6.2]{CiHe}, with appropriate modifications. 

\begin{lemma} \label{lift} If $S$ spans $\coGRaHC_X$, then its image in $\coaHC_X$ spans $\coaHC_X$.\end{lemma}

\begin{proof} We proceed by induction on degree (the base case being precisely $\GRaHC_X$). Commutators preserve degree, and, in particular, if $f_1$ and $f_2$ are homogenous elements of $S(V)$ of degree $k$ and $j$, respectively and $\epsilon_1, \epsilon_2 \in \mathbb{Z}_2^n$, then the top degree term of $[w_1 c^{\epsilon_1} x_1, w_2 c^{\epsilon_2} x_2]$ is given by \begin{equation}\label{degree}y := \pm w_1w_2{w_2}^{-1}x_1x_2w_2^{-1}(c^{\epsilon_1})c^{\epsilon_2} - w_2 w_1 w_1^{-1}(x_2) x_1 w_1^{-1}(c^{\epsilon_2})c^{\epsilon_1}\end{equation} where the signs are determined by the number of nontrivial $c_i$ crossing over $x_i$ terms in $f_1$ or $f_2$. This has degree $j+k$. Hence we have $$[w_1 c^{\epsilon_1} f_1, w_2 c^{\epsilon_2} f_2] \in y + \F^{j+k-1}.$$ 

It suffices to show that we can write homogenous elements of $\aHC_X$ as a linear combination of elements in $S$, commutators, and elements of lower degree. Let $h \in \aHC_X$ be homogenous of degree $k$, and write $h = \sum_{w}  a_ww$, $a_w \in S(V)$. Let $h_0= \sum_w a_w w$ be the corresponding element in $\GRaHC_X$. We have a spanning set for $\GRaHC_X$, so we may write $$h_0 = \sum_{x \in S} c_x x_0$$ where $x$ is the element of $\aHC_X$ represented by $x_0$. Without loss of generality, we may choose these representatives to have maximal degree in $\aHC_X$. Hence we write \begin{align*} h_0 = \sum_{x\in S} c_x x_0 + \sum_i [a_{0,i},b_{0,i}] \quad a_{0,i} b_{0,i} \in \GRaHC_X, [a_{0,i},b_{0,i}] \in \F^k.\numberthis\end{align*}

 Here $a_{0,i}$ and $b_{0,i}$ are representatives of some $a_i,b_i \in \aHC_X$, with $[a_i, b_i] \in \F^k$ for all $i$. By (\ref{degree}), we have $[a_i,b_i] - [a_{0,i},b_{0,i}] \in \F^{k-1}$. Then $h- \sum_{x\in S} c_x x - [a_i, b_i] \in \F^{k-1}$, i.e. the difference between $h$ and its corresponding element in $\GRaHC_X$ has degree less than $k$. Thus we can write $h$ as a linear combination of elements in $S$ up to an element of $\F^{k-1}$; by induction, we are done.
\end{proof}

The following is an immediate consequence of Lemma \ref{lift} and Proposition \ref{thm:spanning}.

\begin{proposition}\label{spanning sets} \begin{enumerate} \item The set $\{w_\lambda f_{J_\lambda;i}\}_{\lambda \in \mathcal{OP}_n}$ spans $\overline{\aHC_{A_{n-1}}}$. 

\item The set $\{w_{\lambda,\mu} f_{J_{\lambda,\mu;i}}\}_{(\lambda ,\mu)\in (\mathcal{OP},\mathcal{EP})}$ spans $\overline{\aHC_{B_{n}}}$. 

\item If $n$ is odd, the set $\{w_{\lambda,\mu} f_{J_{\lambda,\mu;i}}\}_{\shortstack[l]{\tiny$(\lambda ,\mu)\in (\mathcal{OP},\mathcal{EP}),$\\\tiny$ \ell(\mu) $ even}}$ spans $\overline{\aHC_{D_{n}}}$. \\

\noindent If $n$ is even, the set $\{w_{\lambda,\mu} f_{J_{\lambda,\mu;i}}\}$ with $(\lambda ,\mu)\in (\mathcal{OP},\mathcal{EP})$ or $(\lambda,\mu) \in (\emptyset, \mathcal{SOP}_n)$ with $ \ell(\mu)$ even spans $\overline{\aHC_{D_{n}}}$.

\end{enumerate}\end{proposition}

\section{Linear independence in $\aHC_X$}

We now show that our sets from Proposition \ref{spanning sets} are linearly independent, and hence form bases. We first require some results about induction and restriction functors in  $\aHC_X-\text{mod}$ in order to prove a trace formula for parabolically induced $\aHC_X$-modules. The trace formula will allow us to separate the elements of the spanning sets.

 Define the trace pairing $\Tr: \aHC_X \times R(\aHC_X) \rightarrow \CC$ by $$\Tr(h,\pi) =\tr \pi(h).$$

 For each $J\subset I$, let $i_J : (\aHC_X)_J \rightarrow \aHC_X$ be the inclusion. We define $r_J$ as in \cite{CiHe}: for $h\in \aHC_X$, let $\psi_h$ be the right $(\aHC_X)_J$-module morphism given by left multiplication by $h$. Set $r_J(h) = \tr(\psi_h)$ - we can define $\tr \psi_h$ because $\aHC_X$ is a free right $(\aHC_X)_J$-module with finite basis $W^J$. We record two results from \cite{CiHe} needed to prove a trace formula. The following is \cite[Lemma 5.5.1]{CiHe}, and the proof extends easily to the Hecke-Clifford case.

\begin{lemma} Let $J\subset I$. 
\begin{enumerate}
\item For each $h\in (\aHC_X)_J$ and $\pi \in R(\aHC_X)$, we have $\Tr(i_J(h), \pi) = \Tr(h, r_J(\pi))$.
\item For each $h\in \aHC_X$ and $\pi \in R((\aHC_X)_J)$, we have $\Tr(h, i_J(\pi)) = \Tr(r_j(h) , \pi)$.
\end{enumerate}
\end{lemma}

Note that $i_J$ and $r_J$ restrict to well-defined maps $\overline{i_J}$ and $\overline{r_J}$, respectively, on the cocenter; the lemma holds for these maps as well. We also have the following formula for $r_J(wf)$. 

\begin{lemma}\label{lemma:trform} \cite[Proposition 6.3.1]{CiHe} Let $J,J' \subset I$. Let $w\in W_J$ be elliptic and let $C$ be the conjugacy class of $w$ in $W$. Then, for any $f\in S((V^2)^{W_J})^{N_W(W_J)}$, we have $$\overline{r_{J'}}(wf) = \sum_{x\in ^JW^{J'}, x^{-1}(J) \subset J'} x^{-1} \circ (wf).$$

\end{lemma}

\begin{remark}\label{rmk} Note that if $x^{-1}wx \in W_{J'}$ for $x\in W$, we must have that $x\in^JW^{J'}$ and that $x^{-1}(J)\subset J'$ since $w$ is an elliptic element. Conversely, if $x\in^JW^{J'}$ satisfies $x^{-1}(J)\subset J'$, it must also be true that $x w x^{-1} \in W_{J'}$, again because $w$ is elliptic. Hence, if $C\cap W_{J'} =\emptyset$, the above sum is empty, so $r_J(wf)=0$ for any $f$.\end{remark}

\begin{proposition} \label{prop:trform} Let $J,J' \subset I$, let $w\in W_J$ be elliptic and let $C$ be the conjugacy class of $w$ in $W$. Let $M$ be an $(\aHC_X)_{J'}$-module and $f\in S((V^2)^{W_J})^{N_W(W_J)}$. Then we have $$\Tr(wf, \operatorname{Ind}_{(\aHC_X)_{J'}}^{\aHC_X} M) = \left\{\begin{array}{lr}0& \text{if }C\cap W_{J'} = \emptyset \\ |N_W(W_J)/W_J|\Tr(wf,M) &\text{else.} \end{array}\right. $$ \end{proposition}

\begin{proof} If $C\cap W_{J'}=\emptyset$, the statement follows by Lemma \ref{lemma:trform} and Remark \ref{rmk}. Assume $C\cap W_{J'} \not= \emptyset$; since $w$ is elliptic in $W_J$, $J$ must be $J_C$. By Remark \ref{rmk}, there exists an $x\in^JW^{J'}$ such that $x^{-1}(J)\subset J'$, so we must have $|J|\leq |J'|$. Thus $J\sim J'$ by the minimality of $J_C$ with respect to cardinality. By Proposition \ref{prop:normalizers}, $$ |\{z\in^JW^{J'}| z^{-1}(J)=J'\}| = |\{z\in^JW^J|z^{-1}(J)=J\}|= |N_W(W_J)/W_J|$$ where the first equality follows because the second set is sent to the first by $x^{-1}$. The proposition now follows from Lemma \ref{lemma:trform}.\end{proof}

Recall that a conjugacy class $C$ of $W$, $J_C$ is a minimal subset of $I$ such that $C\cap W_{J_C}\not= \emptyset$, $w_C$ is an elliptic element in $W_{J_C}$, and $\{f_{J_C, i}\}$ is a basis of the vector space $S((V^2)^{W_{J_C}})^{N_{W}(W_{J_C})}$. The following is the first main result of the paper. 

\begin{theorem}\label{thm:linind} For $X = A_{n-1}$,  $B_n$, or $D_n$, the spanning set of $\coaHC_X$ given in Proposition \ref{spanning sets} is linearly independent, and hence forms a basis of $\coaHC_X$. \end{theorem}
\begin{proof} We proceed by induction on cardinality of subsets $J\subset I$. Our goal is to apply Proposition \ref{prop:trform} to separate the elements of the spanning set into linearly independent subsets. Suppose that 
$$\sum_{C,i} a_{C,i} w_C f_{J_C;i} = 0$$ where $a_{C,i} \in \CC$. First, set $J= \emptyset$. We have $(\aHC_X)_\emptyset= S(V)$, so every character  of $(\aHC_X)_\emptyset$ is parametrized by an element $v \in V^\vee$. Fix such a $v$ and its corresponding character $\chi_v$, and consider $\ind_{(\aHC_X)_\emptyset}^{\aHC_X} (\chi_v$). Since $W_\emptyset = \{1\}$, by Proposition \ref{prop:trform}, $$\Tr\left(w_C f_{J_C;i}, \ind_{(\aHC_X)_\emptyset}^{\aHC_X} (\chi_v)\right) = 0$$ for all $C\not=\{1\}$ and all $i$. Thus, we have \begin{align*}\Tr\left(\sum_{C,i} a_{C,i} w_C f_{J_C;i},  \ind_{(\aHC_X)_\emptyset}^{\aHC_X} (\chi_v)\right) &= \Tr\left(\sum_i a_{1,i}  f_{\emptyset; i},  \ind_{(\aHC_X)_\emptyset}^{\aHC_X} (\chi_v)\right)  \\
&= |W| \sum_i a_{1,i} (f_{\emptyset; i}, v) \\ &= 0.\end{align*} Hence $\sum_i a_{1,i} f_{\emptyset;i}$ vanishes on each $W$ orbit of $V^\vee$, so the polynomial function $\sum a_{1,i} f_{\emptyset;i}$ vanishes on its entire domain. Thus $$\sum_i a_{1,i} f_{\emptyset, i}=0.$$ But $\{f_{\emptyset; i}\}$ was taken to be a basis. Hence we must have $a_{1,i} = 0$ for all $i$. 

Now let $J\subset I$. We assume by induction that $a_{C',i'} = 0$ for all $i'$ and for every $J_{C'}\subset J$ (up to equivalence with respect to $\sim$). For any module $M$ of $(\aHC_X)_J$, the nonzero summands of $\Tr\left(\sum_{C,i} a_{C,i} w_C f_{J_C;i} , M\right)$ are parametrized by those $C$ such that $J_C \cap J\not= \emptyset$; by the induction hypothesis, we can assume that $J_C= J$ for all such $C$. Let $M$ be an irreducible $(\aHC_X)_J$-module, with irreducible character $\chi_v$ parametrized by $v\in N_W(W_J)/(V^\vee)^{W_J}$. Then applying $\Tr(-,\ind_{(\aHC_X)_J}^{\aHC_X} M)$ to the linear combination gives that \begin{align*}\label{treqn} \Tr\left(\sum_{J_C=J,i} a_{C,i} w_C f_{J_C;i} , \ind_{(\aHC_X)_J}^{\aHC_X}M\right) &= |N_W(W_J)/W_J|\Tr\left(\sum_{J_C=J,i} a_{C,i} w_C f_{J_C;i} , M\right)\\&=|N_W(W_J)/W_J|\sum_{J_C=J, i} a_{C,i}  (f_{J_C;i}, v),\end{align*}  by Proposition \ref{prop:trform}. By hypothesis, we therefore have $$|N_W(W_J)/W_J| \sum_{J_C = J,i} a_{C,i} (f_{J_C,i},v) = 0.$$ As above, this implies that the polynomial $\sum a_{C,i} f_{J_C;i}$ vanishes on its domain, contradicting the linear independence of $\{f_{J_C;i}\}$. Hence, $a_{C,i} = 0$ for all $J_C=J$ (up to equivalence with respect to $\sim$) and all $i$. By induction $a_{C,i} = 0$ for all $C$ and all $i$, as desired.

\end{proof}

\section{Degenerate spin affine Hecke algebras}

We now aim to develop a result on bases of cocenters analogous to Theorem \ref{thm:linind} for a closely related class of algebras, the degenerate spin affine Hecke algebras. 

\subsection{The skew polynomial algebra} Let $\C\langle b_1, \ldots, b_n\rangle$ be the algebra generated by $b_1, \ldots, b_n$ subject to the relations $$b_i b_j + b_j b_i = 0\  (i\not= j).$$ This is the \emph{skew polynomial algebra}. It has a subalgebra $\C\langle b_1^2, \ldots, b_n^2\rangle$; these algebras will take the place of $S(V)$ and $S(V^2)$, respectively, in our discussion of spin affine Hecke algebras.

The skew polynomial algebra has a natural superalgebra structure by letting each $b_i$ be odd.

\subsection{Spin Weyl groups}

Let $W$ be a finite Weyl group. There is a distinguished double cover $\wtd{W}$ for $W$:
\begin{eqnarray}  \label{ext}
1 \longrightarrow \Z_2 \longrightarrow \wtd{W} \longrightarrow W
\longrightarrow 1.
\end{eqnarray}
We denote by $\Z_2 =\{1,z\},$ and by $\td{t}_i$ a fixed preimage
of the generators $s_i$ of $W$ for each $i$. The group $\wtd{W}$
is generated by $z, \td{t}_1,\ldots, \td{t}_n$ with relations

\begin{equation*}
z^2 =1, \qquad
 (\td{t}_{i}\td{t}_{j})^{m_{ij}} =
 \left\{
\begin{array}{rl}
1, & \text{if } m_{ij}=1,3  \\
z, & \text{if }  m_{ij}=2,4,6.
\end{array}
\right.
\end{equation*}

The quotient algebra $\CC W^- :=\CC \wtd{W} /\langle z+1\rangle$ of
$\CC \wtd{W}$ by the ideal generated by $z+1$ is called the
{\emph spin Weyl group algebra} associated to $W$. Denote by $t_i
\in \CC W^-$ the image of $\td{t}_i$. The spin Weyl group algebra
$\CC W^-$ has the following presentation: $\CC W^-$ is the
algebra generated by $t_i, 1\le i\le n$, subject to the relations
\begin{equation}
(t_{i}t_{j})^{m_{ij}} = (-1)^{m_{ij}+1} \equiv \left\{
\begin{array}{rl}
1, & \text{if } m_{ij}=1,3  \\
-1, & \text{if }  m_{ij}=2,4,6.
\end{array}
\right.
\end{equation}
The algebra $\CC W^-$ is naturally a superalgebra by letting each $t_i$ be odd.

In particular, let $W$ be the Weyl group of type $A_{n-1},B_{n},$ or $D_{n}$. Then the spin Weyl group
algebra $\CC W^-$ is generated by $t_1,\ldots, t_n$ with the
labeling as in the Coxeter-Dynkin diagrams (cf. (\ref{coxdynk diagrams})) and the explicit
relations:
%

 \begin{center}
\begin{tabular}
[t]{|l|l|}\hline Type of $W$ & Defining Relations for $\CC W^-$\\
\hline $A_{n-1}$  & $t_{i}^{2}=1$,
$t_{i}t_{i+1}t_{i}=t_{i+1}t_{i}t_{i+1}$,\\ &
$(t_{i}t_{j})^{2}=-1\text{ if } |i-j|\,>1$\\
\hline   & $t_{1},\ldots,t_{n-1}$ satisfy the relations for $\CC W^-_{A_{n-1}}$, \\
$B_{n}$ &  $t_{n}^{2}=1,(t_{i}t_{n})^{2}=-1$ if $i\neq n-1,n$, \\ &
$(t_{n-1}t_{n})^{4}=-1$\\
\hline    & $t_{1},\ldots,t_{n-1}$ satisfy the relations for
$\CC W^-_{A_{n-1}}$,\\
$D_{n}$&  $t_{n}^{2}=1,(t_{i}t_{n})^{2}=-1$ if $i\neq n-2, n$, \\&
$t_{n-2}t_{n}t_{n-2}=t_{n}t_{n-2}t_{n}$\\\hline
\end{tabular}
\end{center}

\subsection{The degenerate spin affine Hecke algebra}

 Let $W= W_{A_{n-1}}$. The degenerate spin affine Hecke algebra of type $A_{n-1}$, $\saH_{A_{n-1}}$, was constructed in \cite{W}. It is the $\CC$-algebra generated by $\C\langle b_1, \ldots, b_{n}\rangle$ and $\CC W^-$ subject to the additional relations: 

\begin{align*}
b_{i+1} t_i + t_i b_i &= 1\\
t_j b_i + b_i t_j &= 0, \quad(i\not=j, j+1).
\end{align*}

Note that $\saH_{A_{n-1}}$ contains the skew polynomial algebra and $\CC W^-$ as subalgebras. The algebra $\saH_{A_{n-1}}$ has a superalgebra structure with all generators being odd.

The degenerate spin affine Hecke algebraas in types $B_n$ and $D_n$ were constructed in \cite{KW}. Let $W= W_{D_n}$. The degenerate spin affine Hecke algebra of type $D_n$, $\saH_{D_n}$, is is the  algebra  generated by $\C\langle b_1,\ldots,b_n\rangle$ and $\CC W^-$
subject to the additional relations:
\begin{align*}
t_{i}b_{i} + b_{i+1}t_i &  =1 \quad (1\leq  i\leq n-1) \\
t_ib_j &=-b_{j}t_i \quad (j\neq i,i+1, \; 1\le i\leq n-1) \\
t_nb_n + b_{n-1}t_n & = 1\\
t_nb_i &=-b_it_n \quad (i \neq n-1,n).
\end{align*} 
In particular, the subalgebra generated by $t_1, \ldots, t_{n-1}$ and $b_1, \ldots, b_n$ is isomorphic to $\saH_{A_{n-1}}$. 

Finally, let $u\in \CC$ and let $W=W_{B_n}$. Then $\saH_{B_n}$ is the algebra generated by $\C\langle b_1,\ldots,b_n\rangle$ and $\CC W^-$ subject to the following
relations:
\begin{align*}
t_{i}b_{i} + b_{i+1}t_i &  =1 \quad (1\leq  i\leq n-1) \\
t_ib_j &=-b_{j}t_i \quad (j\neq i,i+1, \; 1\le i\leq n-1) \\
t_nb_n + b_{n}t_n & =u\\
t_nb_i &=-b_{i}t_n \quad (i\neq n).
\end{align*}

These algebras have a PBW property, and we have a description of their even centers.

\begin{proposition}\cite{W},\cite{KW} \begin{enumerate} \item Let $X= {A_{n-1}},\ {D_n},$ or ${B_n}$. The multiplication of the subalgebras $\CC W^-$ and $\CC[b_1, \ldots, b_n]$ induces a vector space isomorphism $$\C\langle b_1, \ldots, b_n\rangle \otimes \CC W_X^- \xrightarrow{\sim} \saH_X.$$ 

\item Let $X= A_{n-1}, B_n$ or $D_n$. The even center of $\saH_X$ is isomorphic to $\C\langle b_1^2,\ldots, b_n^2\rangle ^{W_X}$.
\end{enumerate}
\end{proposition}

\subsection{A Morita superequivalence}

The degenerate spin affine Hecke algebras are closely related to the degenerate affine Hecke-Clifford algebras, via the following isomorphism.

\begin{proposition}\cite{W},\cite{KW} Let $X= A_{n-1},\ D_n,$ or $B_n$. Then there exists an isomorphism of superalgebras $$\Phi: \aHC_X \xrightarrow{\sim} \C_V \otimes \saH_X.$$ \end{proposition}

Since $\C_V$ is a simple superalgebra, this isomorphism defines a Morita superequivalence in the sense of \cite{W}. In particular, when $n$ is even, $\C_V \cong M(2^{n-1}|2^{n-1})$, and we have a usual Morita equivalence. When $n$ is odd, $\C_V \cong Q(2^{n-1})$, and the categories $\aHC_X$-smod and $\saH_X$-smod are equivalent up to a parity shift.

 In the non-$\ZZ_2$ graded setting, a Morita equivalence $A\rightarrow B$ induces an isomorphism of cocenters $\overline{A}\xrightarrow{\sim} \overline{B}$, because $\overline{X} \cong \HH_0(X)$ for any algebra $X$ (where $HH_*$ is Hochschild homology) and Hochschild homology is Morita-invariant(cf. \cite{Ka}). This result does not extend directly to the superalgebra case - if $n$ is odd, the Morita superequivalence does not necessarily preserve homology - but the superequivalence of $\aHC_X$ and $\saH_X$ nonetheless suggests a connection between their cocenters. 

We aim to compute a basis of $$\displaystyle\overline{\saH_X} = \left(\frac{\saH_X}{[\saH_X,\saH_X]}\right)_{\overline{0}}.$$ using methods similar to the Hecke-Clifford case. As a consequence, we will show that the cocenters are in fact isomorphic in this case.

\section{Reduction for the spin affine Hecke algebra}

This section is analogous to section 3. We prove a variety of reduction results to restrict the types of Weyl group elements that appear in the cocenter.

\subsection{Reduction in type $A_{n-1}$} We adapt the procedure in \cite{WW}, where a basis for the space of trace functions for the spin Hecke algebra in type $A_{n-1}$ is computed. For $w\in S_n$, fix a reduced expression $w= s_{i_1} \ldots s_{i_n}$. Denote $t_{w} = t_{i_1} \ldots t_{i_n} \in \CC W^-$. As in section 3, set $t_{(n)} = t_1 t_2 \ldots t_{n-1}$, and for $\gamma= (\gamma_1, \ldots, \gamma_\ell)$ a composition of $n$, let $t_\gamma = t_{\gamma_1} t_{\gamma_2} \ldots t_{\gamma_\ell}$.

\begin{proposition} Let $\gamma = (\gamma_1, \ldots, \gamma_\ell)$ be a composition of $n$ with $\ell(t_{\gamma})$ even and let $\mu$ be the partition of $n$ corresponding to $\gamma$. Then we have
$$t_{\gamma} \equiv\left\{\begin{array}{lr} 0, &  \text{if }\mu\notin \OP_n \\ \pm t_\mu, &\text{if }\mu \in \OP_n\end{array}\right. \mod\ [\saH_{A_{n-1}}, \saH_{A_{n-1}}].$$  \end{proposition}
\begin{proof}
Suppose $\mu\notin \OP_n$. Let $a$ be the least integer such that $\gamma_a$ is even, and let $b$ be the least integer such that $b>a$ and $\gamma_b$ is even; such a $b$ must exist because $\ell(t_\gamma)$ is even. Set $t_{y, k} = t_{\gamma_1 + \ldots + \gamma_{k-1} +1} \ldots t_{\gamma_1 + \ldots + \gamma_k}$ (the cycle corresponding to $\gamma_k$ in $t_\gamma$). Thus $t_{\gamma} = t_{\gamma,1}\ldots t_{\gamma,\ell}$. Commuting $t_{\gamma,k}$ over $t_{\gamma,j}$ results in a sign of $(-1)^{(j-1)(k-1)}$, which is negative only if $j$ and $k$ are both even. Thus we have
\begin{align*} t_\gamma &\equiv t_{\gamma, a} t_{\gamma,a+1} \ldots t_{\gamma,\ell} t_{\gamma,1} \ldots t_{\gamma, a-1} \\
&= t_{\gamma,a} t_{\gamma_b} t_{\gamma, a+1} \ldots t_{\gamma,a-1}\\
&=-t_{\gamma, b} t_{\gamma, a} \ldots t_{\gamma,a-1} \\
&\equiv -t_{\gamma,a} \ldots t_{\gamma,a-1} t_{\gamma,b}\\
&= -t_\gamma
\end{align*}

where the equivalences are mod $[\saH_{A_{n-1}}, \saH_{A_{n-1}}]$. Hence $t_\gamma \equiv 0$ mod $[\saH_{A_{n-1}}, \saH_{A_{n-1}}]$. 

If $\mu \in \OP_n$, the images of $t_\gamma$ and $t_\mu$ in $\overline{\aHC_{A_{n-1}}}$ are equal; since the isomorphism $\Phi$ restricts to an injective map on $\overline{\CC W^-}$, they must be equal in $\overline{\saH_{A_{n-1}}}$ as well.
\end{proof}

\begin{proposition}\label{comp=part} Let $w_C$ be a minimal length representative of a conjugacy class $C$ corresponding to the cycle type $\mu= (\mu_1, \ldots, \mu_\ell) \vdash n$. Then we have $$t_{w_C} \equiv \left\{ \begin{array}{lr} \pm t_{\mu}, &\text{ if }\mu\in \OP_n \\ 0, & \text{ otherwise} \end{array}\right. \mod\ [\saH_{A_{n-1}}, \saH_{A_{n-1}}].$$\end{proposition}

\begin{proof} The minimal length representative must have the form $$w_C = (s_{i_1^1} s_{i_2^1} \ldots s_{i_{\gamma_1 -1}^1})(s_{i_1^2} s_{i_2^2} \ldots s_{i_{\gamma_2-1}^2}) \ldots(s_{i_1^\ell} \ldots s_{i_{\gamma_\ell-1}^\ell})$$ where $\gamma = (\gamma_1, \ldots, \gamma_\ell)$ is a composition of $n$ given by rearragning the parts of $\mu$, and $i_{j}^k = \gamma_1 + \ldots + \gamma_{k-1} + j$. We claim that $t_{w_C} \equiv \pm t_\gamma \mod [\saH_{A_{n-1}}, \saH_{A_{n-1}}]$; the lemma will then follow from Proposition \ref{comp=part}. 

It suffices to consider the case $\gamma = (n)$ (by dealing with each cycle separately). Thus $w_C = s_{i_1}\ldots s_{i_{n-1}}$. If $i_j = j$ for all $1\leq j \leq n-1$, then $w_C = w_\gamma$. Otherwise, there is at least one $a$ such that $i_a \not=a$; choose the smallest such $a$ (note that we must have $i_a > a$). We proceed by induction on $a$. Observe that \begin{align*} t_{w_C} &= (-1)^{a-1}t_{i_a} t_1 t_2 \ldots t_{a-1} t_{i_{a+1}} \ldots t_{i_{n-1}} \\ &\equiv (-1)^{a-1} t_1 t_2 \ldots t_{a-1} t_{i_{a+1}} \ldots t_{i_{n-1}} t_{i_a}  \ \mod [\saH_{A_{n-1}}, \saH_{A_{n-1}}]\numberthis\label{conjclass eqn}\end{align*} Now $t_{i_{a+1}}$ is in the $a$th position; repeat this process until $t_{i_{a+k}} = t_a$ is in the $a$th position. 

Hence $$t_{w_C} = \pm t_{1} t_2 \ldots t_{a} t_{i_{a+1}'} \ldots t_{i_{n-1}'}.$$ By the inductive hypothesis, we are done. \end{proof}


\subsection{Reduction in types $B_n$ and $D_n$}

Set $t_{(n)}^- = t_1 t_2 \ldots t_{n-1} t_n = t_{(n)} t_n$. For compositions $\gamma = (\gamma_1, \ldots, \gamma_\ell)$ and $\nu = (\nu_1, \ldots, \nu_k)$ of $n$, let $t_{(\gamma, \nu)} = t_{\gamma_1} \ldots t_{\gamma_{\ell}} t_{\nu_1}^- \ldots t_{\nu_k}^-$. 

\begin{proposition}\begin{enumerate}\item Let $\gamma = (\gamma_1, \ldots, \gamma_\ell)$ and $\nu = (\nu_1, \ldots, \nu_k)$ be compositions of $n$ with $\ell(t_{\gamma})+\ell(t_{\nu})$ even and let $(\lambda,\mu)$ be the bipartition of $n$ corresponding to $(\gamma,\nu)$. Then we have
$$t_{(\gamma,\nu)} \equiv\left\{\begin{array}{lr} 0, &  \text{if }(\lambda,\mu)\notin (\OP, \EP) \\ \pm t_{(\lambda,\mu)}, &\text{if }(\lambda,\mu) \in (\OP,\EP)\end{array}\right. \mod\ [\saH_{B_{n}}, \saH_{B_{n}}].$$ 
\item Let $\gamma,\nu,\lambda$, and $\mu$ be as in the previous part. If $n$ is odd, we have$$t_{(\gamma,\nu)} \equiv\left\{\begin{array}{lr} 0, &  \text{if }(\lambda,\mu)\notin (\OP, \EP) \\ \pm t_{(\lambda,\mu)}, &\text{if }(\lambda,\mu) \in (\OP_n,\EP_n)\end{array}\right. \mod\ [\saH_{D_{n}}, \saH_{D_{n}}].$$ If $n$ is even, we have$$t_{(\gamma,\nu)} \equiv\left\{\begin{array}{lr} \pm t_{(\lambda,\mu)}, &  \text{if }(\lambda,\mu)\in (\OP, \EP) \text{ or} (\lambda, \mu) \in (\emptyset, \SOP_n) \\  0, &\text{ otherwise}\end{array}\right. \mod\ [\saH_{D_{n}}, \saH_{D_{n}}].$$  \end{enumerate}\end{proposition}

\begin{proof}\begin{enumerate}\item If $\lambda \notin \OP$, we can repeat the proof of Proposition 7.1.1 to show that $t_\gamma \equiv 0$, and hence $t_{(\gamma, \nu)} \equiv 0$. Suppose $\mu\notin \EP$. Let $a$ be the smallest integer so that $\nu_a$ is odd. If $\ell(\nu)$ is even, let $b$ be the smallest integer such that $b>a$ and $\nu_b$ is odd. Then we have, as in Proposition 7.1.1, \begin{align*} t_{\nu}^- &\equiv t_{\nu,a}^- t_{\nu, a+1}^-\ldots t_{nu,\ell}^-   t_{\nu,1}^- \ldots t_{\nu, a-1}^- \\
&= -t_{\nu, b}^- t_{\nu, a}^- \ldots t_{\nu, a-1}^-\\
&\equiv (-1)^{2(k-b-1) +1 } t_{\nu}^-. \end{align*}
Here the extra signs come from commuting $t_{\nu, b}^-$ past the $t_n$ in each term. Hence $t_{\nu}^- \equiv 0$ mod $[\saH_X, \saH_X]$. 

If $(\gamma, \nu) \in (\OP, \EP)$ the equality follows as in the type $A_{n-1}$ case.

\item If $n$ is odd or $n$ is even and $(\lambda, \mu) \notin (\emptyset, \SOP_n)$, the proof follows as in type $B_n$. For $(\lambda, \mu) \in \SOP_n$, following the proof as in type $B_n$ gives 
\begin{align*} t_{\nu}^- &\equiv t_{\nu,a}^- t_{\nu, a+1}^-\ldots t_{nu,\ell}^-   t_{\nu,1}^- \ldots t_{\nu, a-1}^- \\
&= -t_{\nu, b}^- t_{\nu, a}^- \ldots t_{\nu, a-1}^-\\
&\equiv (-1)^{1+2(2)} t_{\nu,a}\ldots t_{\nu,b} t_n t_{n-2} t_n \ldots t_{\nu, a-1}\\
&= (-1)^{5} t_{\nu, a} \ldots t_{\nu,b} t{n-2} t_{n} t_{n-2} \ldots t_{\nu, a-1}\\
&=  (-1)^{k-b-1+5} t_{\nu}^-.\end{align*} 
But since $\ell(\mu)$ is even, $k-b-1$ is odd, so $k-b-1+5$ is even. Hence we have no sign change. Indeed, since the image of $t_{\nu}^-$ is nonzero and equal to $t_{\mu}$ in $\aHC_{D_n}$, we must have the same equality in $\saH_{D_n}$.   \end{enumerate} \end{proof}

\begin{proposition}\begin{enumerate} \item  Let $w_C$ be a minimal length representative of a conjugacy class $C$ corresponding to the bipartition $(\lambda, \mu)$ in $W_{B_n}$. Then we have $$t_{w_C} \equiv \left\{ \begin{array}{lr} \pm t_{(\lambda,\mu)}, &\text{ if }(\lambda,\mu)\in (\OP, \EP) \\ 0, & \text{ otherwise,} \end{array}\right. \mod\ [\saH_{B_{n}}, \saH_{B_{n}}].$$ 
\item  Let $w_C$ be a minimal length representative of a conjugacy class $C$ corresponding to the bipartition $(\lambda, \mu)$ in $W_{D_n}$. Then if $n$ is odd we have $$t_{w_C} \equiv \left\{ \begin{array}{lr} \pm t_{(\lambda,\mu)}, &\text{ if }(\lambda,\mu)\in (\OP, \EP) \\ 0, & \text{ otherwise.} \end{array}\right. \mod\ [\saH_{D_{n}}, \saH_{D_{n}}].$$ If $n$ is even, we have $$t_{w_C} \equiv \left\{\begin{array}{lr} \pm t_{(\lambda,\mu)}, &  \text{if }(\lambda,\mu)\in (\OP, \EP) \text{ or} (\lambda, \mu) \in (\emptyset, \SOP_n) \\  0, &\text{ otherwise.}\end{array}\right.  \mod\ [\saH_{D_{n}}, \saH_{D_{n}}].$$  \end{enumerate} \end{proposition}
\begin{proof} Both cases follow immediately from the proof of Proposition 7.1.2, using Proposition 7.2.1 in place of 7.1.1 and modifying Equation \ref{conjclass eqn} to use the appropriate relations.\end{proof} 

\section{Bases for the cocenter of the spin affine Hecke algebra}
We establish spanning sets for $\cosaH_X$, and then use the Morita superequivalence to prove that these sets are linearly independent. 
\subsection{Spanning sets}
 Let $\mathsf{C}_X$ be the set of conjugacy classes labelled by $\OP_n$ if $X= A_{n-1}$, by $(\OP, \EP)$ if $X=B_n$, and by $(\OP, \EP)$ if $X= D_n$ with $n$ odd, or $(\OP, \EP) \cup (\emptyset, \SOP_n)$ if $X=D_n$ with $n$ even. The following is similar to \cite[Theorem 6.6]{WW}.
\begin{lemma}\label{lincomblemma} Let $w\in W_X$ with $\ell(w)$ even. Then there exist $f_{w,\mathbf{\nu}}^- \in \CC$ such that $$t_w \equiv \sum_{\mathbf{\nu} \in \mathsf{C}_X} f_{w,\nu}^- t_{\nu}.$$ \end{lemma}
\begin{proof} The type $A_{n-1}$ case is proved in \cite{WW}. The type $B_n$ and $D_n$ cases follow from a similar argument using Proposition 7.2.2.   \end{proof}

As before, filter $\saH_X$ by degree in $\C\langle b_1, \ldots, b_n\rangle$ and let ${\saHo_X}$ be the associated graded object, which is isomorphic to the degenerate spin affine Hecke algebra with all parameters identically 0. Now we follow the procedure in Section 4. 

\begin{lemma} For $X= A_{n-1}$, $B_n$, or $D_n$, we have $$\cosaHo_X \subset \operatorname{span}\{t_{w_C} \CC[b_1^2, \ldots, b_n^2]\}_{C\in \mathsf{C}_X}.$$ \end{lemma} \begin{proof} Apply Propositions 7.1.2 and 7.2.2 to each element in the (trivial) spanning set $W_X^- \C\langle b_1^2, \ldots, b_n^2\rangle$. Thus every element is either congruent to 0 or to $t_{w_C}$ mod $[\saH_X, \saH_X]$. \end{proof}

For a conjugacy class $C$ of $W^-$, define $J_C$ and $w_C$ as before, using the natural action of $W^-$ on $\{1, 2, \ldots, n\}$. Now, for convenience, denote $\C\langle\mathbf{b}^2\rangle = \C\langle b_1^2, \ldots, b_n^2\rangle$. 
\begin{lemma} Fix a conjugacy class $C$ of $W$, and let $J= J_C$. Then we have $$t_{w_C}\C\langle \mathbf{b}^2\rangle \equiv t_{w_C} \C\langle (\mathbf{b}^2)^{W_J^-}\rangle^{N_{W^-}(W_J^-)}.  $$ \end{lemma}
\begin{proof} The proof of Proposition 4.1.4 extends to this case without modification except possible the addition of signs: it depends only on the action of $W_X$ on $S(V)$, which is the same as the action of $W_X^-$ on $\C\langle \mathbf{b}^2\rangle$ with parameter 0 up to a possible change in sign. \end{proof}

For each $C\in \mathsf{C}_X$ and $W$ of types $A_{n-1}$, $B_n$, or $D_n$, let $\{f_{J_C; i}^-\}$ be a basis of the vector space $\C\langle(\mathbf{b}^2)^{W_J^-}\rangle^{N_{W^-}(W_J^-)}$. Combining Lemmae 8.0.2 and 8.0.3 gives: 

\begin{proposition} For $X= A_{n-1}$, $B_n$, or $D_n$, we have $$\cosaHo_X = \operatorname{span}\{t_{w_C}   f_{J_C; i}^-\}_{C\in \mathsf{C}_X}.$$ \end{proposition}

Finally, we lift the spanning set to the ungraded object, as before. 

\begin{proposition}\label{spin spanning set}  For $X= A_{n-1}$, $B_n$, or $D_n$, we have that $$\overline{\saH_X} = \operatorname{span}\{t_{w_C}  f_{J_C;i}^- \}_{C\in \mathsf{C}_X}.$$\end{proposition}
\begin{proof} Equations 22 and 23 lift to the spin case up to a change in sign; but these equations were already agnostic to sign, so this does not affect the proof. Hence any spanning set of $\cosaHo_X$ is also a spanning set of $\overline{\saH_X}$. \end{proof}

\subsection{Linear independence}
The following is the second main result of the paper.
\begin{theorem} For $X= A_{n-1}$, $B_n$, or $D_n$, the set $\{t_{w_C}f_{J_C;i}^-\}_{C\in C_X}$ forms a linear basis of $\overline{\saH_X}$.\end{theorem}

\begin{proof} It suffices to prove that these sets are linearly independent. We take advantage of Theorem 5.0.2 and the isomorphism in Proposition 6.4.1. The inverse $\Phi^{-1}$ of the isomorphism in Proposition 6.4.1 induces an injective map $\saH_X \rightarrow \aHC_X$. We claim that it restricts to an inclusion $$\overline{\Phi^{-1}}: \overline{\saH_X} \rightarrow{\quad}  \coaHC_X.$$ Indeed, by Proposition \ref{spin spanning set}, the set $\{t_{w_C}f_{J_C;i}^-\}_{C\in \mathsf{C}_X}$ spans $\cosaH_X$. The image of an element in this set under $\Phi^{-1}$ is $$\Phi^{-1}(t_{w_C} f^-) = w_C f$$ where $f \in S(V^2)$ is obtained from $f^- \in \C\langle \mathbf{b}^2\rangle$ by replacing all $b_i$'s with $x_i$'s. But the images of the elements $w_C f_{J_C;i}$ for $C\in \mathsf{C}_X$ were shown to be linearly independent in Theorem 5.0.2, so the map $\overline{\Phi^{-1}}$ is an inclusion. But the set $\{t_{w_C} f_{J_C;i}^-\}_{C\in \mathsf{C}_X}$ is then the preimage under an inclusion of a linearly independent set, and must therefore be linearly independent.

\end{proof}


\end{document}